\documentclass{amsart}

\usepackage{xcolor}

\usepackage{amsthm, amsmath, amssymb}

\usepackage{mathtools, bm, dsfont, mathrsfs, mhsetup, cases, stmaryrd, wasysym, bbm}

\usepackage{enumitem}

\usepackage{caption, chngcntr, setspace}
\usepackage{textcomp, microtype, csquotes, ragged2e, verbatim}

\usepackage{array, multicol, multirow}

\usepackage{graphicx, wrapfig}
\usepackage{tikz, tikz-cd, pgfplots}
\pgfplotsset{compat=1.18}
\usepackage{tikzsymbols}

\usepackage{hyperref, cleveref}
\usepackage{amsrefs}
\hypersetup{
  colorlinks=true,
  linkcolor=blue,
  filecolor=magenta,
  urlcolor=cyan
}

\newtheorem{theorem}{Theorem}[section]
\newtheorem{prop}[theorem]{Proposition}
\newtheorem{lemma}[theorem]{Lemma}

\theoremstyle{remark}
\newtheorem*{remark}{Remark}

\counterwithout{equation}{section}

\newcommand{\ii}[1]{\llbracket{#1}\rrbracket}

\newcommand{\bE}{\mathbb E}

\newcommand{\bN}{\mathbb N}
\newcommand{\N}{\mathbb N}
\newcommand{\bP}{\mathbb P}

\newcommand{\bR}{\mathbb R}
\newcommand{\bZ}{\mathbb Z}
\newcommand{\Z}{\mathbb Z}

\newcommand{\Instr}{\mathsf{Instr}}
\newcommand{\Left}{\texttt{Left}}
\newcommand{\Right}{\texttt{Right}}
\newcommand{\Sleep}{\texttt{Sleep}}

\newcommand{\s}{\mathfrak{s}}

\newcommand{\idla}{\omega_{\mathrm{IDLA}}}
\newcommand{\ind}{\mathbbm{1}}
\newcommand{\pss}{\bP_n}
\newcommand{\rhoPS}{\rho_{\mathrm{PS}}}
\newcommand{\rhoDD}{\rho_{\mathrm{DD}}}

\newcommand{\cC}{\mathcal C}

\newcommand{\eps}{\varepsilon}

\newcommand{\crit}{\rho_c}
\newcommand{\eqd}{\overset{d}{=}}
\newcommand{\cz}{\cC_n^{(0)}}
\newcommand{\co}{\cC_n^{(1)}}

\newcommand{\pnk}{P_n(k)}

\DeclareMathOperator{\TV}{TV}
\DeclareMathOperator{\Stab}{Stab}

\begin{document}

\title{Local density of activated random walk on $\Z$}

\author[Christopher~Hoffman]{Christopher~Hoffman}
\address{Department of Mathematics, University of Washington, Seattle, WA 98195}
\email{hoffman@math.washington.edu}

\author[Jacob~Richey]{Jacob~Richey}
\address{Alfréd Rényi Institute of Mathematics, Budapest, Hungary}
\email{jrichey@renyi.hu}

\author[Hyojeong~Son]{Hyojeong~Son}
\address{Department of Mathematics, University of Washington, Seattle, WA 98195}
\email{hjson@math.washington.edu}

\begin{abstract}
We consider one-dimensional activated random walk (ARW) on $\bZ$ started from a `point source' initial condition, with many particles at the origin and no other particles. We prove that, uniformly throughout a macroscopic window around the source, the probability that a site contains a sleeping particle after the configuration is stabilized is approximately the critical density. This represents a first step towards understanding the local structure of the critical stationary measure for ARW. 
\end{abstract}

\maketitle

\section{Introduction}

Many complex, real-world systems share a common pattern: they accumulate energy gradually and release it in sudden, cascading bursts, which are statistically scale-free. This phenomenon arises in earthquake dynamics, where stress accumulates progressively along geological fault lines before dissipating abruptly through seismic events of varying magnitudes, from negligible tremors to catastrophic ruptures. Similarly, financial markets display this behavior through price fluctuations in stocks and commodities. These observations led Bak, Tang, and Wiesenfeld in 1987 to introduce the unifying notion of \textit{self-organized criticality} (SOC) \cite{BakTangWiesenfeld87}. The defining characteristic of SOC is that the system is autonomously driven toward a critical state without fine-tuning of parameters or dependence on the initial condition. The prevalence of this mechanism across both artificial and natural phenomena underscores its importance as a universal framework for understanding the behavior of complex systems.

Since then, many mathematical toy models have been introduced as candidates to capture and rigorously define SOC. One such model is \textit{activated random walk} (ARW). ARW is an interacting particle system on $\bZ^d$ whose particles are either active or sleeping. Active particles perform independent continuous-time simple random walks at rate $1$ and attempt to fall asleep at rate $\lambda\in(0,\infty)$; an attempt succeeds only if the particle is alone at its site. Sleeping particles do not move and are instantaneously reactivated when an active particle arrives. Starting from finitely many particles, the dynamics stabilize almost surely in finite time.

ARW shares many common features with its spiritual predecessor, the abelian sandpile model, specifically its abelian property. From a modeling and technical perspective, ARW has two important advantages over the abelian sandpile model: first, the fact that individual particles perform random walks and become sleeping through independent mechanisms; and second, that ARW is expected to exhibit SOC in a robust way, while the abelian sandpile lacks some of the universality properties associated with SOC \cite{FeyLevineWilson10,Levine15}.

There are several natural notions of criticality for ARW, including the fixed-energy, driven-dissipative, cycle, and point-source formulations. Hoffman, Johnson, and Junge proved that these four definitions agree in $d=1$ \cite{HoffmanJohnsonJunge24}. We write $\crit\in(0,1)$ for this common critical density. We next recall the point-source and driven-dissipative critical densities in the formulation used in \cite{HoffmanJohnsonJunge24}.

\medskip
\noindent\textbf{Point-source model.}
Start ARW on $\bZ$ from the initial condition $n\delta_0$, consisting of $n$ active particles at the origin and no other particles, and let $\eta^{(n)}=\Stab(n\delta_0)\in\{0,\s\}^{\bZ}$ be the final stabilized configuration (we define the stabilization operator $\Stab$ in Section~\ref{sec:sitewise-ARW}). Let $L_n=L_n(\eta^{(n)})$ be the number of sites in the smallest integer interval containing all sleeping particles in $\eta^{(n)}$, and define the point-source critical density
\[
\rhoPS \;:=\; \lim_{n\to\infty}\bE\!\left[\frac{n}{L_n}\right],
\]
whenever the limit exists.

\medskip
\noindent\textbf{Driven-dissipative model.}
Fix $m\in\bN$ and consider ARW on the finite interval $\ii{0,m}:=\{0,1,\dots,m\}$ with sinks at $-1$ and $m+1$. Consider the Markov chain on stable configurations in which, at each step, one adds a single active particle and then stabilizes until only sleeping particles remain in $\ii{0,m}$. The active particle may be added uniformly at random in $\ii{0,m}$ or at a fixed site; this choice does not affect the stationary distribution of the chain \cite{levine2021exact}. Let $\pi_m^{\mathrm{DD}}$ be the stationary distribution. If $\eta\sim\pi_m^{\mathrm{DD}}$, let
\[
N_m(\eta)\;:=\;\sum_{x=0}^m \mathbf 1_{\{\eta(x)=\s\}}
\]
be the number of sleeping particles in $\ii{0,m}$, and define
\[
\rhoDD \;:=\; \lim_{m\to\infty}\frac{1}{m+1}\,\bE_{\pi_m^{\mathrm{DD}}}[N_m],
\]
whenever the limit exists.

By \cite{HoffmanJohnsonJunge24}, these limits exist and coincide, $\rhoPS=\rhoDD=\crit$. We use the point-source formulation in the statement of our main theorem, while the driven-dissipative viewpoint enters in Section~\ref{sec:block}.

So far, little progress has been made on a key aspect of ARW: the existence of a critical stationary distribution in the infinite-volume setting which exhibits the hallmarks of SOC. Depending on the choice of context, this critical measure arises in a few different possible forms, but the same critical object, which can be thought of as a probability measure on particle configurations, should be visible in each one. The limit measure is expected to exhibit the statistical features which constitute a definition of SOC, namely multi-scale power laws for correlation decays and avalanche sizes, and hyperuniformity, which is a reduction in the variance of the number of particles in a large region compared to independent randomness.

One way to define the critical measure, which is relevant for our main result, is presented by Levine and Silvestri in \cite[Section 2.3]{LevineSilvestriUniversality24}. There, it is conjectured to arise as the microscopic limit of ARW started from a large \emph{point-source} initial condition. Namely, for each $n$ let $\pss$ denote the law (on the space of sleeping particle configurations $\{0,\s\}^{\mathbb{Z}^d}$) of the final sleeping configuration of ARW on $\mathbb Z^d$ started from $n$ active particles at the origin. It is conjectured in \cite{LevineSilvestriUniversality24} that $\pss$ converges locally, in the sense that for every finite set $V\subset\mathbb Z^d$ the restrictions $\pss|_V$ converge as $n\to\infty$, and that the resulting limiting measure $\alpha$ is translation invariant and has a well-defined particle density $\rho_c(d,\lambda)$. Alternatively, the measures $\pi_m^{\mathrm{DD}}$ should converge locally to the same limit $\alpha$ as $m \to \infty$. 

For the abelian sandpile model on $\bZ^d$ with $d\geq 2$, an infinite-volume limit of the stationary measure exists \cite{AthreyaJarai04}. Local probabilities under this limiting measure can sometimes be computed explicitly, especially in two dimensions where the burning bijection relates sandpiles to spanning trees. For example, exact height probabilities on $\bZ^2$ (that is, the probability that the stable configuration has a given number of particles at the origin) are available in special cases \cite{Priezzhev94}. In high dimensions, the single-site height distribution admits a Poisson-type asymptotic as $d\to\infty$ \cite{JaraiSun21}, and related computations have been carried out on self-similar graphs including the Sierpi\'nski gasket \cite{HeizmannKaiserSavaHuss25}. These spanning tree techniques are not available for ARW, so different methods are needed.

Our main result identifies the single-site occupation probability of any subsequential local limit measure for point-source ARW. Concretely, for any site $i$ not too far from the source, the probability that $i$ contains a sleeping particle in the final configuration is approximately $\crit$.

\begin{theorem}\label{thm:key}
Fix $\eps>0$. For any sequence of sites $i(n)\in \mathbb Z$ with
\[
|i(n)|\leq \frac{1}{2}\bigl(1-\eps\bigr)n,
\]
we have
\[
\pss\bigl(\eta(i(n))=\s\bigr) = \crit + o(1)\quad \text{as } n\to\infty,
\]
where $\eta\sim\pss$ is the final sleeping configuration of point-source ARW started from $n\delta_0$.
\end{theorem}

\noindent As a consequence of our methods, specifically those in Section~\ref{sec:IDLA}, we obtain that the limiting critical measure, if it exists, is shift invariant.

\begin{prop}\label{cola}
Let $\alpha$ be a probability measure on $\{0,\s\}^{\mathbb Z}$ which is a subsequential limit of the point-source laws $\pss$. Then $\alpha$ is shift invariant.
\end{prop}

To summarize, we find that for every fixed site $i\in\mathbb Z$ the one-site marginals $\pss|_{\{i\}}$ converge and the limiting probability that $i$ contains a sleeping particle is $\crit$. In particular, if the full microscopic limit measure $\alpha$ exists, it must be shift-invariant with average particle density $\crit$. These results confirm part of Conjectures 3--5 of~\cite{LevineSilvestriUniversality24}.

\subsection{Outline and Overview}
For $i\in\mathbb Z$, let $S_i:=\{\eta^{(n)}(i)=\s\}$ be the event that site $i$ contains a sleeping particle in the final stabilized configuration. Equivalently, if $\eta\sim\pss$, then $S_i=\{\eta(i)=\s\}$ and
\[
\pss\bigl(\eta(i)=\s\bigr)=\bP(S_i).
\]
The idea of the proof is to show that (1) the function $i\mapsto \pss(\eta(i)=\s)$ is relatively flat, and (2) that over any block of size $\omega(1)$ as $n \to \infty$, the particle density is close to $\crit$. We prove (1) in Section \ref{sec:IDLA} using a coupling between two instances of internal diffusion limited aggregation (IDLA), which can be thought of as ARW with infinite sleep rate. (2) is proved in Section \ref{sec:block} by an application of the machinery of~\cite{HoffmanJohnsonJunge24}, in combination with a semi-artificial toppling sequence involving IDLA on the $\omega(1)$-sized block. The artificial topplings allow us to give an upper bound on the probability that the particle density in the block is atypical. These two uses of IDLA are distinct and different from those in the existing ARW literature. Finally, in Section \ref{sec:complete} we combine (1) and (2) with elementary inequalities to obtain Theorem \ref{thm:key}.

\section{ARW Setup} \label{sec:sitewise-ARW}

We follow the same conventions as most recent work on ARW by using the so-called site-wise (Diaconis--Fulton) construction of the process. Namely, to every site we attach an infinite stack of instructions, and evolve the system by applying the instructions to the particles that arrive at that site. Thanks to the abelian property (Lemma \ref{camellia}), as long as the resulting configuration is stable (i.e. contains only empty sites and sleeping particles), the order in which sites are toppled is irrelevant. We give an abbreviated version of the setup, with particular focus on the `flattening' toppling sequences that will be relevant for our proofs. We refer the reader to~\cite{rolla2020activated} for more details.

\medskip

\noindent\textbf{State space and configurations.}
We take $\mathbb N=\{0,1,2,\dots\}$. Let $\mathbb{N}\cup\{\s\}$ be ordered by $0<\s<1<2<\cdots$. A \emph{configuration} is $\eta\in\{0,\s,1,2,\dots\}^{\bZ}$, with $\eta(x)$ the state at $x\in\bZ$: $\eta(x)=0$ means no particle, $\eta(x)=\s$ means one sleeping particle, and $\eta(x)=n\ge1$ means $n$ active particles. When an active particle arrives at a site containing a sleeping particle, it wakes it, so $\s+1=2$. We set $\lvert \s\rvert=1$ and write $\lvert \eta(x)\rvert$ for the number of particles at $x$. A site $x$ is \emph{stable} if $\eta(x)<1$ and \emph{unstable}
otherwise.

\medskip

\noindent\textbf{Instruction stacks and their execution.}
Let $(\Omega,\mathcal F,\bP)$ be the product space of instruction stacks, under which $(\Instr_x(k))$ are i.i.d. with the following distribution
\[
  \Instr_x(k)=
  \begin{cases}
    \Left  & \text{with prob.\ }\frac{1}{2(1+\lambda)},\\
    \Right & \text{with prob.\ }\frac{1}{2(1+\lambda)},\\
    \Sleep & \text{with prob.\ }\frac{\lambda}{1+\lambda}.
  \end{cases}
\]
A $\Left$ (resp.\ $\Right$) instruction removes $1$ from $\eta(x)$ and adds $1$ to $\eta(x-1)$ (resp.\ $\eta(x+1)$). A $\Sleep$ instruction at a site with exactly one active particle changes $1\mapsto\s$, while it has no effect if there are $k\ge2$ active particles. We write $\bE$ for expectation with respect to $\bP$.

\medskip

\noindent\textbf{Topplings and odometer.}
One should think of each stack $(\Instr_x(k))_{k \in \N^{+}}$ as an infinite roll of train tickets; then particles arrive one at a time at the train station, and are given the next ticket in the roll to read and execute. We call each such execution a {\em{toppling}} at a given site. Given a finite sequence of sites $\kappa=(x_1,\dots,x_\ell)$, we may execute the {\em{toppling sequence}} which topples those sites in order. The \emph{odometer} of a toppling sequence records how many times each site has been toppled:
\[
U_{\kappa}:\bZ\to\mathbb{N},\qquad U_{\kappa}(x)=\text{(\# of topplings at $x$ in $\kappa$).}
\]
If some instructions have already been executed, and $(\eta, U)$ is the current state and partially-executed toppling sequence odometer, the next state and odometer after toppling site $x$ is
\[
\Phi_{x}(\eta,U) = \bigl(\Instr_x(U(x)+1)(\eta),\, U+\delta_x\bigr),
\]
where $\delta_x$ is $1$ at $x$ and $0$ elsewhere. The move $\Phi_x$ is \emph{legal} for $(\eta, U)$ when $x$ is unstable for $\eta$. For a toppling sequence $\kappa=(x_1,\dots,x_\ell)$, write
\[
(\Phi_\kappa(\eta),U_\kappa)\;:=\;\Phi_{x_\ell}\circ\cdots\circ\Phi_{x_1}(\eta,0),
\]
so that $\Phi_\kappa(\eta)$ denotes the resulting configuration. We call a toppling sequence $\kappa=(x_1,\dots,x_\ell)$ \emph{legal} for an initial configuration $\eta$ if for each $j=1,\dots,\ell$, the move $\Phi_{x_j}$ is legal for the state obtained after performing the first $j-1$ topplings of $\kappa$ starting from $(\eta,0)$. We call $\kappa$ \emph{stabilizing} for $\eta$ if the resulting configuration $\Phi_\kappa(\eta)$ is stable (i.e., every site is stable). The abelian property shows that all legal stabilizing schemes are equivalent:

\begin{lemma}[Abelian Property {\cite{rolla2020activated}}]\label{camellia}
Let $\kappa$ and $\beta$ be legal, stabilizing toppling sequences for a configuration $\eta$. Then $U_\kappa(x)=U_\beta(x)$ for every $x\in\bZ$, and in particular,
\[
\Phi_\kappa(\eta)=\Phi_\beta(\eta).
\]
\end{lemma}

Fix a realization of the instruction stacks $(\Instr_x(k))$. For a configuration $\eta$ which admits a finite legal stabilizing toppling sequence, define the \emph{stabilization} of $\eta$ by
\[
\Stab(\eta)\;:=\;\Phi_\kappa(\eta),
\]
where $\kappa$ is any legal stabilizing toppling sequence for $\eta$. By Lemma~\ref{camellia}, $\Stab(\eta)$ does not depend on the choice of $\kappa$. We similarly write
\[
u(\eta)\;:=\;U_\kappa
\]
for the associated (well-defined) odometer function. For a finite set $V$, we write $\Stab_V(\eta)$ for the stabilization of the configuration $\eta$ when legal topplings are allowed only at sites in $V$, and sites in $V^{\mathrm c}$ are treated as sinks (equivalently, particles that move outside $V$ are removed (ignored)). We write $u_V(\eta,x)$ for the associated odometer, that is, the number of topplings performed at site $x \in V$ during the stabilization of $\eta$ in $V$ (and set $u_V(\eta,x)=0$ for $x \notin V$). We reserve $\Stab(\cdot)$ (without a subscript) for the global stabilization map defined earlier; finite-volume stabilizations always carry the subscript $V$.

\subsection{Point-source ARW}
Fix a positive integer $n$, and consider the initial condition $n\delta_0$ consisting of $n$ active particles at the origin. Let
\[
\eta^{(n)} \;:=\; \Stab(n\delta_0)\in\{0,\s\}^{\bZ}
\]
be the resulting stabilized configuration. All randomness in this paper comes from the i.i.d.\ instruction stacks, and we work on the stack space $(\Omega,\mathcal F,\bP)$ introduced earlier. For each site $i\in\bZ$, define the event
\[
S_i \;:=\; \{\eta^{(n)}(i)=\s\}\in\mathcal F,
\]
that site $i$ contains a sleeping particle after stabilization. Its probability is $\bP(S_i)$. When it is convenient to view $\eta^{(n)}$ as a random element of $\{0,\s\}^{\bZ}$, we denote by
\[
\pss \;:=\; \bP\circ(\eta^{(n)})^{-1}
\]
the induced law on final sleeping configurations.

\subsection{IDLA Toppling Sequence}\label{IDLA:flatphase}

We will use a specific type of toppling sequence, which is legal and stabilizing, in Sections \ref{sec:IDLA} and \ref{sec:block}. The toppling sequence works in two phases: a {\em{flattening}} phase, where particles perform IDLA steps, and then a second phase, where particles can be toppled in any legal order until the configuration is stable. For the sake of definiteness, the flattening phase evolves in stages by always toppling the leftmost site $x$ where $|\eta(x)| \geq 2$, so this phase ends when only sites with $|\eta(x)| \leq 1$ remain. Note that during this phase, any sleep instruction encountered is effectively voided, since the particle which executed the sleep is at a site with at least two particles. Thus, when started from the point-source initial condition $n\delta_0$ on $\bZ$, the flattening phase is equivalent to performing IDLA (aka ARW with infinite sleep rate), except that it results in an interval of length $n$ with each site containing a single active particle (whereas IDLA would leave a single sleeping particle at each site). In particular, by Lemma \ref{camellia} for ARW with $\lambda = \infty$, all legal choices of flattening procedure yield the same intermediate configuration $\idla$ obtained at the end of the flattening phase.


\section{Coupling two one-dimensional IDLA clusters}\label{sec:IDLA}

We now prove a coupling result for two IDLAs. (Recall that IDLA is just ARW with $\lambda = \infty$, i.e.\ where particles perform independent simple random walks, and fall asleep instantly whenever they are alone.) Let $\cC_n^{(s)}$ denote the IDLA cluster obtained by releasing $n$ particles from the site $s\in\{0,1\}$ and performing IDLA until all particles are asleep. We will construct an explicit coupling of the two clusters $\cC_n^{(0)}$ and $\cC_n^{(1)}$ so that they coincide with high probability as $n \to \infty$.

Let $K_n^{(s)}$ denote the number of occupied sites strictly to the right of $s$ in $\cC_n^{(s)}$ after stabilization. Note that the random variable $K_n^{(s)}$ completely determines $\cC_n^{(s)}$ (for each $s \in \{0,1\}$, respectively) because the IDLA cluster is always a connected interval in $\bZ$. It follows that the two clusters $\cC_n^{(0)}$ and $\cC_n^{(1)}$ coincide if and only if
\[
K_n^{(0)}=K_n^{(1)}+1.
\]
Thus, to show the desired coupling, it is equivalent to construct a coupling of the laws $K_n^{(0)}$ and $K_n^{(1)}$ for which $K_n^{(0)}=K_n^{(1)}+1$ holds with high probability.

Let $D_n$ denote the number of descents of a uniformly random permutation $\sigma\in S_n$, that is
\[
D_n=\#\{1\leq i\leq n-1:\sigma(i)>\sigma(i+1)\}.
\]
The following result of Mittelstaedt gives an exact description of the law of $K_n^{(s)}$ via the Eulerian distribution of descents $D_n$.

\begin{theorem}[Theorem 1 of \cite{Mittelstaedt19}]\label{thm:eulerian}
For each source $s\in\{0,1\}$,
\[
K_n^{(s)}\eqd D_n,\qquad
P_n(k):=\frac{\langle n,k\rangle}{n!},\quad 0\le k\le n-1,
\]
where $\langle n,k\rangle$ is the $k$th Eulerian number.
\end{theorem}

In particular, $\Pr(D_n=k)=P_n(k)=\langle n,k\rangle/n!$ for $0\le k\le n-1$. We extend $P_n$ to a probability mass function on $\bZ$ by setting $P_n(k)=0$ for $k\notin\{0,1,\dots,n-1\}$, and define its right shift
\[
Q_n(k):=P_n(k-1),\qquad k\in\bZ.
\]
Recall that $\cz=\co$ if and only if $K_n^{(0)}=K_n^{(1)}+1$. By maximal coupling, the optimal success probability for coupling a random variable with law $P_n$ to one with law $Q_n$ is $1-\TV(P_n,Q_n)$, where
\[
\TV(P_n,Q_n):=\frac{1}{2}\sum_{k\in\bZ}\bigl|P_n(k)-Q_n(k)\bigr|
\]
is the total variation distance. Since $K_n^{(0)}\sim P_n$ and $K_n^{(1)}+1\sim Q_n$, there exists a coupling of $K_n^{(0)}$ and $K_n^{(1)}$ such that $K_n^{(0)}=K_n^{(1)}+1$ with probability at least $1-\TV(P_n,Q_n)$. Therefore it suffices to bound $\TV(P_n,Q_n)$. We use the following standard fact about the Eulerian distribution: the Eulerian numbers are log-concave in $k$ \cite{Stanley89}, and hence the sequence $\pnk$ is unimodal in $k$. From this, a short telescoping argument gives an exact identity.

\begin{lemma}\label{lem:comet}
If $(P(k))_{k\in\bZ}$ is unimodal and $Q(k)=P(k-1)$, then
\[
\TV(P,Q)=\max_kP(k).
\]
\end{lemma}

\begin{proof}
Write $a_k=P(k)-Q(k).$ Since $\sum_k a_k=0$ and the sign of $a_k$ is nonnegative up to the mode and negative thereafter,
\[
\TV(P,Q)=\frac{1}{2}\sum_k|a_k|=\sum_k(a_k)_+=\max_t\sum_{k\leq t}a_k=\max_tP(t),
\]
because $\sum_{k\leq t}a_k=P(t)$ by telescoping.
\end{proof}

Applying Lemma \ref{lem:comet} to $P_n$ and $Q_n$ yields
\[
\TV(P_n,Q_n)=\max_k\pnk.
\]
Thus the coupling problem reduces to bounding the maximal point mass of $D_n$.

Let $F_n(x):=\sum_{k\le \lfloor x\rfloor} P_n(k)$. Let $\Phi_n$ denote the CDF of $\mathcal N\left(\mu_n,\sigma_n^2\right)$ with
\[
\mu_n=\frac{n-1}{2},\quad \sigma_n^2=\frac{n+1}{12}.
\]
By Lemma~\ref{lem:comet}, \(\TV(P_n,Q_n)=\max_k P_n(k)\), so it remains to control the maximal atom of the Eulerian distribution. We obtain an \(O(n^{-1/2})\) bound on \(\max_k P_n(k)\) from a uniform normal approximation due to \"Ozdemir.
\begin{theorem}[Theorem 1.1 of \cite{Ozdemir22}]\label{thm:ozdemir}
\[
\lVert F_n - \Phi_n \rVert_{\infty}
   \;:=\;
   \sup_{x \in \mathbb{R}} \lvert F_n(x) - \Phi_n(x) \rvert
   \;\le\;
   C\, n^{-1/2}.
\]
\end{theorem}

\begin{lemma}\label{lem:someday}
\[
\max_k\pnk\leq \left(\sqrt{\frac{6}{\pi}}+2C\right)n^{-1/2}.
\]
\end{lemma}

\begin{proof}
Since $\pnk=F_n(k)-F_n(k-1)$, we obtain
\[
\pnk\leq (\Phi_n(k)-\Phi_n(k-1))+ 2\|F_n-\Phi_n\|_\infty
\leq \max_{t\in\bR}\left(\Phi_n(t)-\Phi_n(t-1)\right)+\frac{2C}{\sqrt n}.
\]
And
\[
\max_{t\in\bR}\left(\Phi_n(t)-\Phi_n(t-1)\right)\leq
\frac{1}{\sqrt{2\pi}\sigma_n}
=\sqrt{\frac{6}{\pi}}\cdot\frac{1}{\sqrt{n+1}}
\leq\sqrt{\frac{6}{\pi}}\cdot\frac{1}{\sqrt n}.
\]
Therefore
\[
\max_k\pnk\leq\left(\sqrt{\frac{6}{\pi}}+2C\right)n^{-1/2}.
\]
\end{proof}

Combining Lemmas \ref{lem:comet} and \ref{lem:someday}, and following the IDLA toppling sequence as laid out in Section \ref{IDLA:flatphase}, shows that the probabilities $\pss(\eta(i)=\s)$ and $\pss(\eta(i+1)=\s)$ are close:

\begin{lemma}\label{lem:skittle}
There exists $c > 0$ such that for all $n \in \N$ and all $i \in \bZ$,
\begin{equation*}
\bigl|\pss(\eta(i)=\s) - \pss(\eta(i+1)=\s)\bigr| \leq c\, n^{-1/2}.
\end{equation*}
\end{lemma}

\begin{proof}
Applying Lemmas \ref{lem:comet} and \ref{lem:someday} gives
\[
\TV(P_n,Q_n)=\max_k\pnk\leq cn^{-1/2}
\]
for a global constant $c > 0$. By maximal coupling, there exists a coupling of $\cC_n^{(0)}$ and $\cC_n^{(1)}$ such that $\cC_n^{(0)}=\cC_n^{(1)}$ with probability at least $1-\TV(P_n,Q_n)$. This shows that we can couple $\cC_{n}^{(s)}$, $s \in \{0,1\}$, so that they are equal with probability at least $1-cn^{-1/2}$.

Now consider two instances of point-source ARW on $\bZ$ with $n$ particles, one having source $0$ and the other source $1$. We couple these two instances by applying the flattening procedure described in Section \ref{IDLA:flatphase} simultaneously to both instances. Specifically, for the IDLA phase, we use the coupling guaranteed by the previous paragraph. If the resulting IDLA clusters $\cC_n^{(s)}$, $s \in \{0,1\}$ are the same, then we couple the remaining (unused) portions of the instruction stacks to be identical and complete stabilization using these coupled stacks; otherwise, stabilize the two instances independently. In the former case, the final configurations are identical by construction. The result follows.
\end{proof}

As a corollary, we obtain that any subsequential local limit of the measures $\pss$ must be shift invariant. 

\begin{proof}[Proof of Proposition \ref{cola}]
Let $(\theta\eta)(x)=\eta(x-1)$ denote the shift operator on configurations, and for a measure $\mu$ write $\theta\mu$ for its pushforward under $\theta$. Let $\pss^{(1)}$ be the point-source law on $\{0,\s\}^{\bZ}$ for $n$ particles started at $1$; by translation invariance of the stacks, $\pss^{(1)}=\theta\,\pss$. The coupling constructed in the proof of Lemma~\ref{lem:skittle} gives $\TV(\pss,\pss^{(1)})\le cn^{-1/2}$, hence $\TV(\pss,\theta\,\pss)\le cn^{-1/2}$. Let $n_k\to\infty$ be a subsequence along which $\pss\rightarrow\alpha$. For any cylinder event $E$,
\[
|\pss(E)-\pss(\theta^{-1}E)|\le \TV(\pss,\theta\,\pss)\le cn^{-1/2},
\]
so evaluating at $n=n_k$ and letting $k\to\infty$ yields $\alpha(E)=\alpha(\theta^{-1}E)$. Thus $\alpha$ is shift invariant.
\end{proof}


\section{Block averages near the source}\label{sec:block}

In this section we show that the average particle density over any block of growing size in the bulk is asymptotically equal to the critical density $\crit$. Fix $\eps,\gamma \in (0,1)$. For each $n$, we choose a site $i=i(n)\in\bZ$ satisfying
\begin{equation}\label{clementine}
\lvert i(n)\rvert \leq \frac{1}{2}\left(1-\eps\right)n.
\end{equation}
We then consider the block
\begin{equation*}
    I_n=\ii{i-\lfloor n^\gamma\rfloor,\,i}.
\end{equation*}
The bulk condition \eqref{clementine} is chosen so that, with high probability, the one–dimensional IDLA cluster produced by $n$ particles at the origin contains the entire block $I_n$ and leaves exactly one particle at each site of $I_n$ (see Lemma \ref{lem:bots} below).

We evolve the point-source system using the toppling procedure described in Section \ref{IDLA:flatphase}. We now briefly recall the procedure and add some auxiliary notation.

\emph{Phase 1 (IDLA flattening).}  We follow the {\em{flattening}} toppling procedure described in Section \ref{IDLA:flatphase}. This produces a configuration $\idla$. 

\emph{Phase 2 (ARW stabilization).} Starting from $\idla$ we resume the full ARW dynamics on $\bZ$, by performing any legal, stabilizing toppling sequence.

During Phase 2 we define the boundary fluxes
\begin{align*}
    a^\ast &:= \#\{\text{$\Right$ instructions used at the site } i-\lfloor n^\gamma\rfloor-1 \text{ during Phase 2}\},\\
    b^\ast &:= \#\{\text{$\Left$ instructions used at the site } i+1 \text{ during Phase 2}\}.
\end{align*}
In other words, $a^\ast$ and $b^\ast$ are the total numbers of jumps entering $I_n$ from the left and from the right, respectively, during Phase 2. Let $D$ be the number of particles in $I_n$ in the final stabilized configuration of the full ARW on $\bZ$.

We will use the following auxiliary driven-dissipative process which is coupled to the evolution in phase 2. For any $a,b\ge 0$, we consider the following finite-volume system associated with $I_n$: run ARW on the interval $I_n$ with sinks at $i-\lfloor n^\gamma\rfloor-1$ and $i+1$, started from the restriction of $\idla$ to $I_n$ with $a$ additional active particles placed at $i-\lfloor n^\gamma\rfloor$ and $b$ additional active particles placed at $i$. We use the same instruction stacks at sites in $I_n$ as in the full system. Let $D_{a,b}$ denote the number of particles in $I_n$ when this finite-volume system stabilizes. Our first goal is to relate $D$ to $D_{a^\ast,b^\ast}$, so that the block density can be analyzed via a
driven-dissipative system.

\begin{lemma}\label{yoshi}
Fix $n$ and consider the two--phase construction above. Then, for every realization of the instruction stacks, we have
\[
D = D_{a^\ast,b^\ast}.
\]
\end{lemma}

\begin{proof}
Write $I_n=[L,R]$ where $L=i-\lfloor n^\gamma\rfloor$ and $R=i$. We now work with the instruction stacks as they stand at the beginning of Phase~2, after Phase~1 has already consumed some instructions.

Throughout this proof, for any configuration $\eta$ on $I_n$, $\Stab_{I_n}(\eta)$ denotes the stabilized configuration obtained by starting from $\eta$ and applying any legal, stabilizing toppling sequence that topples only sites in $I_n$ (equivalently, $L-1$ and $R+1$ act as sinks). By the abelian property (Lemma \ref{camellia}) restricted to $I_n$, $\Stab_{I_n}(\eta)$ is well defined and independent of the choice of legal stabilizing sequence in $I_n$. For $x\in I_n$, define the addition operator
\[
A_x(\eta) := \Stab_{I_n}(\eta+\delta_x).
\]
It follows from Lemma \ref{camellia} that for all $x,y\in I_n$ and all configurations $\eta$ on $I_n$,
\begin{equation}\label{winter}
A_xA_y = A_yA_x
\quad\text{and}\quad
A_x(\Stab_{I_n}(\eta)) = A_x(\eta),
\end{equation}
and more generally that for any $a,b\ge0$,
\begin{equation}\label{winds}
A_L^a A_R^b(\eta) = \Stab_{I_n}\bigl(\eta + a\delta_L + b\delta_R\bigr).
\end{equation}

Now consider Phase~2 of the full system on $\bZ$, started from $\idla$. We follow the following $I_n$-first toppling procedure: at each step
\begin{itemize}
    \item if there is an unstable site in $I_n$, topple one such site;
    \item otherwise, topple an unstable site in $\bZ\setminus I_n$.
\end{itemize}
This defines a legal stabilizing toppling sequence on $\bZ$, so by Lemma \ref{camellia} it produces the same final configuration and odometer as any other legal stabilizing sequence.

Let $(\eta_t)_{t\ge0}$ be the sequence of configurations during this $I_n$-first stabilization in Phase~2. Let
\[
\tau_1<\tau_2<\dots<\tau_M
\]
be the times $t$ at which a particle jumps into $I_n$ from outside. By definition of $a^\ast$ and $b^\ast$,
\[
M = a^\ast + b^\ast.
\]

Define configurations $\zeta_j$ on $I_n$ as follows. Let $\zeta_0$ be the restriction to $I_n$ of the first configuration $\eta_t$ for which every site in $I_n$ is stable. Hence,
\[
\zeta_0 = \Stab_{I_n}(\idla\vert_{I_n}).
\]
For $j\ge1$, at time $\tau_j$ a particle enters $I_n$ at some boundary site $s_j\in\{L,R\}$. Between the last time at which $I_n$ was stable (whose restriction is $\zeta_{j-1}$ by definition) and the instant $\tau_j$, we topple only outside $I_n$, so the configuration on $I_n$ just before $\tau_j$ is still $\zeta_{j-1}$. And immediately after the jump at time $\tau_j$ the configuration on $I_n$ is $\zeta_{j-1}+\delta_{s_j}$. After time $\tau_j$ we again topple only sites in $I_n$ until every site in $I_n$ is stable; let $\zeta_j$ be the resulting configuration on $I_n$. Then,
\[
\zeta_j = \Stab_{I_n}(\zeta_{j-1}+\delta_{s_j}) = A_{s_j}(\zeta_{j-1}),\qquad
j\ge1.
\]
Iterating, we obtain
\begin{equation}\label{sunshine}
\zeta_M = A_{s_M}\cdots A_{s_1}(\zeta_0).
\end{equation}
Among the entry sites $s_1,\dots,s_M$, the site $L$ appears exactly $a^\ast$ times and the site $R$ appears exactly $b^\ast$ times, so by the commutativity in \eqref{winter},
\[
A_{s_M}\cdots A_{s_1} = A_L^{a^\ast}A_R^{b^\ast}.
\]
Combining this with \eqref{sunshine}, the identity $\zeta_0=\Stab_{I_n}(\idla\vert_{I_n})$, and the relation $A_x(\Stab_{I_n}(\eta)) = A_x(\eta)$ from \eqref{winter}, we obtain
\[
\zeta_M
  = A_{s_M}\cdots A_{s_1}(\zeta_0)
  = A_L^{a^\ast}A_R^{b^\ast}\bigl(\zeta_0\bigr)
  = A_L^{a^\ast}A_R^{b^\ast}\bigl(\Stab_{I_n}(\idla\vert_{I_n})\bigr)
  = A_L^{a^\ast}A_R^{b^\ast}\bigl(\idla\vert_{I_n}\bigr).
\]

By construction of the $I_n$--first procedure, $\zeta_M$ is exactly the restriction to $I_n$ of the final configuration of the full ARW on $\bZ$ at the end of Phase~2, since after time $\tau_M$ there are no further jumps into $I_n$ and we never topple inside $I_n$ again once it is stable. Therefore $D$, the number of particles in $I_n$ in the final stabilized configuration of the full system, is exactly the number of particles in $\zeta_M$.

On the other hand, consider the finite-volume ARW on $I_n$ with sinks at $L-1$ and $R+1$, started from the configuration $\idla\vert_{I_n}+a^\ast\delta_L+b^\ast\delta_R$, using the same instruction stacks at sites in $I_n$ as at the beginning of Phase~2. By definition of $\Stab_{I_n}$ and by the identity \eqref{winds}, the final stabilized configuration of this finite system is
\[
\Stab_{I_n}\bigl(\idla\vert_{I_n}+a^\ast\delta_L+b^\ast\delta_R\bigr)
   = A_L^{a^\ast}A_R^{b^\ast}\bigl(\idla\vert_{I_n}\bigr)
   = \zeta_M.
\]
Thus the number of particles in $I_n$ in this finite-volume system, which is $D_{a^\ast,b^\ast}$ by definition, coincides with $D$.
\end{proof}

We now show that the average particle density in $I_n$ converges to $\crit$.

\begin{lemma}\label{lem:bots}
For any $\eps, \gamma \in (0,1)$ and integer $i$ satisfying \eqref{clementine},
we have
\begin{equation}\label{trombone}
\sum_{j=i-\lfloor n^\gamma \rfloor}^{i} \pss\bigl(\eta(j)=\s\bigr)
= (\crit+o(1))\, n^{\gamma}
\qquad \text{as } n \to \infty.
\end{equation}
\end{lemma}

\begin{proof}
In the final configuration each site contains at most one sleeping particle. So by linearity of expectation
$$
\bE[D]=\sum_{j\in I_n}\pss\bigl(\eta(j)=\s\bigr).
$$
Therefore it suffices to show $\bE[D]=\left(\crit+o(1)\right)n^\gamma.$ We first perform Phase 1 as outlined at the beginning of this section. Let $A_1$ be the event that, after this phase, there is not exactly one active particle at each site of $I_n$. By Theorem 1 of \cite{Mittelstaedt19} we can deduce that there exist constants $c_1,C_1>0$ such that
\begin{equation}\label{practice}
    \bP(A_1)\leq C_1e^{-c_1n^{.5}}.
\end{equation}
On $A_1^c$ the IDLA configuration $\idla$ has exactly one active particle at
each $j\in I_n$.

Next, we run Phase 2. By Lemma \ref{yoshi} we have $D=D_{a^\ast,b^\ast}$. For each fixed $a,b\geq 0$, consider the finite-volume ARW on $I_n$ with sinks at $i-\lfloor n^\gamma\rfloor-1$ and $i+1$. In this system we start from the restriction of $\idla$ to $I_n$ (where each site $j\in I_n$ contains one active particle), then add $a$ additional active particles at the left boundary site $i-\lfloor n^\gamma\rfloor$ and $b$ additional active particles at the right boundary site $i$, and then stabilize. Theorem 2.1 of \cite{levine2021exact} implies that the law of the resulting stabilized configuration is the stationary distribution of the driven-dissipative ARW on $I_n$ with these sinks and boundary driving. Fix $\delta>0$, and for integers $a$ and $b$ with $a,b\ge 0$, define the deviation event
$$
A_{2,a,b}:=\left\{D_{a,b}\not\in\left(n^\gamma\crit\left(1-\delta/2\right),\,n^\gamma\crit\left(1+\delta/2\right)\right)\right\}.
$$
By Proposition 8.5 and 8.6 of \cite{HoffmanJohnsonJunge24}, we know there exist constants $c_2,C_2>0$ such that for all $a,b\geq 0$ and all sufficiently large $n$,
$$
\bP(A_{2,a,b})\leq C_2 e^{-c_2n^\gamma}.
$$
Next we control the size of the boundary fluxes. By the same argument as Lemma 3.5 of \cite{BrownHoffmanSon24} there exist constants $c_3, C_3>0$ such that
$$
\bP(a^\ast >n^5 \text{ or } b^\ast>n^5)\leq C_3 e^{-c_3n}.
$$
We now define a good event. Let $A$ be the event that
\begin{enumerate}
    \item $A_1^c$ occurs,
    \item $\Bigl(\bigcup_{a,b\in\ii{0,n^5}} A_{2,a,b}\Bigr)^c$ occurs, and
    \item $a^\ast$ and $b^\ast$ are between $0$ and $n^5$.
\end{enumerate}
By a union bound with the fact that there are at most $n^{10}$ pairs $(a,b)$, we know that there exist constants $c,C,\beta>0$ such that
$$
\bP(A^c)\leq Cn^{10}e^{-cn^\beta}.
$$
Also note that $0\leq D\leq n$ and on $A$
$$
D=D_{a^\ast,b^\ast}\in\left(n^\gamma\crit\left(1-\delta/2\right),\, n^\gamma \crit\left(1+\delta/2\right)\right).
$$
Now we are ready to bound $\bE[D]$. The lower bound is
$$
\bE[D]\geq\bE[D \ind_A] \geq \bP(A)\,n^\gamma\crit\left(1-\delta/2\right).
$$
The upper bound is
$$
\bE[D]\leq \bP(A)\,n^\gamma\crit\left(1+\delta/2\right) +\bP(A^c)\,n\leq n^\gamma\crit\left(1+\delta\right).
$$
Combining these two equations above we get that
$$
\bE[D]=n^\gamma\crit\left(1+o(1)\right)
$$
as desired.
\end{proof}


\section{Completing the Proof}\label{sec:complete}

\begin{proof}[Proof of Theorem \ref{thm:key}]
By Lemma \ref{lem:skittle}, for any $j \in \bZ$ we have
\begin{equation*}
  \bigl|\pss(\eta(j)=\s) - \pss(\eta(j+1)=\s)\bigr| \leq c\, n^{-1/2}.
\end{equation*}
By the triangle inequality this implies that for all $i,j\in\bZ$,
\begin{equation}\label{eq:golf}
  \bigl|\pss(\eta(i)=\s) - \pss(\eta(i-j)=\s)\bigr| \leq c\, j\, n^{-1/2}.
\end{equation}
Summing over $j \in [0, \lfloor n^{\gamma}\rfloor] \cap \bZ$ and using
\eqref{eq:golf}, we obtain
\begin{equation*}
  \left|
    \sum_{j=i-\lfloor n^{\gamma} \rfloor}^{i} \pss(\eta(j)=\s)
      - \lfloor n^{\gamma} \rfloor \pss(\eta(i)=\s)
  \right|
  \leq \sum_{j=0}^{\lfloor n^{\gamma} \rfloor} c\, j\, n^{-1/2}
  \leq c'\, n^{2\gamma - 1/2}.
\end{equation*}
Now take $i=i(n)$ as in the statement of the theorem. By Lemma~\ref{lem:bots} with the same $\eps$ and $\gamma$, we have
\[
  \sum_{j=i-\lfloor n^{\gamma} \rfloor}^{i} \pss(\eta(j)=\s)
    = (\crit+o(1))\,n^{\gamma}.
\]
Combining this with the previous inequality and dividing by $n^{\gamma}$ gives
\[
  \bigl|\pss(\eta(i)=\s) - \crit\bigr|
    \le c\, n^{\gamma-1/2} + o(1),
\]
which tends to $0$ as $n\to\infty$ by taking any $0<\gamma<1/2$. This proves the claim.
\end{proof}

\begin{remark}
Note that the source of the $o(1)$ error in Theorem \ref{thm:key} is the proof of Lemma \ref{lem:bots}, and ultimately Lemmas 8.5 and 8.6 of \cite{HoffmanJohnsonJunge24}, where $\delta$ depends only implicitly on $n$. Thus, our method does not give a quantitative estimate for the error term. 
\end{remark}


\section{Future Directions}\label{sec:bladerunner}

Our results verify the one–site part of Conjecture~3 of \cite{LevineSilvestriUniversality24} in the one–dimensional point–source model, conditional on the existence of the full microscopic limit: if the point–source  measures converge locally, then the limiting measure is shift invariant and has density $\crit$. A natural next step would be to extend our methods to multi–site events, for example to show that the joint law of the sleeping indicator field $(\mathbf 1_{S_i})_{i\in\bZ}$ converges on finite windows, and to relate its covariance structure to the hyperuniformity conjecture in~\cite{LevineSilvestriUniversality24}.

Two further directions concern changing either the geometry or the driving mechanism. First, it is very plausible that an analogue of Theorem~\ref{thm:key} should hold in two and higher dimensions for the point–source model on $\bZ^d$. Our proof suggests that this would follow from an IDLA statement about the harmonic measure on the boundary of a large ball: namely, that the hitting distribution is essentially unchanged when the source is moved from the origin to a neighbor of the origin. Second, one expects a local density theorem of the same form to hold for the driven–dissipative ARW on large finite intervals. Our attempted coupling breaks down in this setting because particles can fall into the sinks at different times in the two coupled systems, so new ideas seem necessary. Establishing the analogues of Theorem~\ref{thm:key} and Corollary~\ref{cola} in these settings would give further evidence for universality of the limit measure.


\section{Acknowledgements}

J.R. was partially supported by a Simons Foundation Targeted Grant awarded to the HUN-REN Alfréd Rényi Institute of Mathematics when this research was conducted. He is currently supported by the Hungarian National Research, Development and Innovation Office, grant 152849. C.H. and H.S. were partially supported by NSF grant DMS-1954059. C.H. was also partially supported by NSF grant DMS- 2503778.


\bibliographystyle{amsalpha}
\bibliography{main}

\end{document}